\documentclass[11pt,reqno]{amsart}
\newtheorem{theorem}{Theorem}[section]
\newtheorem{proposition}[theorem]{Proposition}
\newtheorem{corollary}[theorem]{Corollary}
\newtheorem{remark}[theorem]{Remark}
\newtheorem{lemma}[theorem]{Lemma}

\newtheorem{definition}[theorem]{Definition}

\usepackage{amssymb}
\usepackage{tikz}
\usetikzlibrary{matrix,arrows,decorations.pathmorphing}
\usepackage{graphicx}
\usepackage{amsmath}
\numberwithin{equation}{section}
\usepackage[all]{xy}
\usepackage{tikz-cd}
\begin{document}

\title[]{new examples of reducible theta divisors for some Syzygy bundles}
\author{Abel Castorena and H. Torres-L\'opez}

\address{Centro de Ciencias Matem\'aticas. Universidad Nacional Aut\'onoma de M\'exico, Campus Morelia. Antigua carretera P\'atzcuaro 8701. C.P. 58089. Colonia: Ex-hacienda de San Jos\'a de la Huerta. Morelia, Michoac\'an,M\'exico. } 
\email{abel@matmor.unam.mx}
\email{hugo@matmor.unam.mx}
\thanks{
The second named author is supported with a Posdoctoral Fellowship from CONACyT}




\date{\today}

\begin{abstract}
Let $C$ be a smooth complex irreducible projective curve of genus $g$ with general moduli, and let $(L,H^0(L))$ be a generated complete linear series of type $(d,r+1)$ over $C$. The syzygy bundle, denoted by $M_L$, is the kernel of the evaluation map $H^0(L)\otimes\mathcal O_C\to L$. In this work we have a double purpose. The first one is to give new examples of stable syzygy bundles admitting theta divisor over general curves. We prove that if $M_L$ is strictly semistable then $M_L$ admits reducible theta divisor. The second purpose is to study the cohomological semistability of $M_L$, and  in this direction we show that when $L$ induces a birational map, the syzygy bundle $M_L$ is cohomologically semistable, and we obtain precise conditions for the cohomological semistability of $M_L$ where such conditions agree with the semistability conditions for $M_L$.

\end{abstract}
\maketitle
 \section{Introduction}
 Let $C$ be a smooth complex irreducible projective curve of genus $g$. Denote by $S\mathcal{U}(r,M)$, the moduli space of semistable vector bundles of rank $r$ with fixed determinant $M\in Pic^d(C)$ over $C$. Assume that the slope of $M$, $\mu=\frac{d}{r}$, is an integer number and set $v:=g-1-\mu$. Consider a line bundle  $N$ in $Pic^{v}(C)$, define 
\begin{equation}
\mathcal D_N:=\{E\in S\mathcal{U}(r,M)| h^0(E\otimes N)>0\}.
\end{equation}
It is well known that $\mathcal D_N$ describes a Cartier divisor on  $S\mathcal{U}(r,M)$. The associated line bundle $\mathcal{L}:=\mathcal{O}(\mathcal D_N)$ does not depend on the choice of $N$. This line bundle is called the {\it determinant line bundle}, moreover, the Picard group of $S\mathcal{U}(r,M)$ is generated by $\mathcal{L}$ i.e., $Pic(S\mathcal{U}(r,M))=\mathbb{Z}\cdot\mathcal{L}$. 

\noindent For a vector bundle $E\in S\mathcal{U}(r,M)$, the {\it theta divisor} for $E$ is defined as
\begin{eqnarray*}
 \Theta_{E}:=\{\mathcal{P}\in Pic^v(C)| h^0(\mathcal{P}\otimes E)\neq 0\}\subset Pic^v(C).
\end{eqnarray*} 

\noindent The Strange duality (see e.g. \cite{popastrange}, Section 5.2) says that there is a canonical isomorphism
 \begin{eqnarray*}
 SD: H^0(S\mathcal{U}(r,M),\mathcal{L})^{\vee}\rightarrow H^0(Pic^{v}(C),\mathcal{O}(r\Theta)),
 \end{eqnarray*}
 making the following diagram commutative
 \begin{eqnarray}\label{strange}
 \xymatrix{ S\mathcal{U}(r,M)  \ar@{-->}[dr]_{\theta}\ar@{-->}[r]^{\phi_{\mathcal{L}}}
&|\mathcal{L}|^{\vee} \ar[d]^{|\wr}\\
 & |r\Theta|}
 \end{eqnarray}
 where $\Theta:=\{ \mathcal{P}\in Pic^{v}(C) | h^0(\mathcal{P}\otimes M)\neq 0\}$ and	$\theta$ is the {\it theta map} which associates to a vector bundle $E$ its theta divisor $\Theta_E$. 
 We say that $E$ {\it admits theta divisor} if $\Theta_E\subsetneq Pic^v(C)$, in this case  $\Theta_E$ has a natural structure of divisor in $Pic^v(C)$. The points where the map $\theta$ is not defined  correspond to vector bundles that do not admit theta divisor, and such bundles can be identified with the base points of the linear system $|\mathcal{L}|$.

\vspace{.2cm}

We recall some properties of theta divisors:
\begin{itemize}
 \item  If $E$ admits theta divisor, then $E$ is semistable: suppose that there is a subbundle $F\subset E$ with $\mu(F)>\mu(E)$, thus by Riemann Roch we have $h^0(F\otimes\mathcal{P})>0$ for all $\mathcal{P}\in Pic^v(C)$, hence $h^0(E\otimes\mathcal{P})>0$ and this implies that $E$ does not admit theta divisor.
\item  If $E$ admits theta divisor, then $E^{\vee}$ admits theta divisor.
\item If $E$ admits theta divisor, then $E\otimes L$ admits theta divisor for any line bundle $L$.
\end{itemize}

\noindent A generated linear series of type $(d,r+1)$ over $C$ is a pair $(L,V)$, where $L$ is a generated line bundle of degree $d$ on $C$ and $V\subseteq H^0(L)$ is a linear subspace of dimension $r+1$ that generates $L$. The kernel $M_{V,L}$ of the evaluation map $V\otimes \mathcal{O}_C\rightarrow L$ fits into the following exact sequence	
\begin{equation}\label{dualspam}
 \xymatrix{ 0 \ar[r]^{} &  M_{V,L} \ar[r]^{}&  V\otimes \mathcal{O}_C  \ar[r]^{} &  L \ar[r]^{} & 0.}
\end{equation}

\noindent The bundle $M_{V,L}$ is called a syzygy bundle. When $V=H^0(L)$, we will denote the bundle $M_{H^0(L),L}$ by $M_L$. The vector bundle $M_{V,L}$ and its dual $M_{V,L}^{\vee}$ have been studied from different points of view because of the rich geometry they encode. For example, the syzygy bundle is important in the study  of Brill-Noether varieties for curves, the Maximal Rank Conjecture and the Minimal Resolution Conjecture (see e.g., \cite{resolucionminimal}). For example, in \cite{resolucionminimal}, the authors prove that the vector bundle  $M_{K_C}$ and it's exterior powers $\wedge^2 M_{K_C},\ldots, \wedge^{g-2}M_{K_C}$ admit theta divisor, where $K_C$ is the canonical line bundle over $C$. In \cite{beaville}, Beauville proves that  for a non-hyperelliptic curve and a general line bundle $L$ of degree $2g$, the vector bundle $M_L$ and it's exterior powers $\wedge^2 M_L,\ldots, \wedge^{g-1}M_L$ admit a reducible theta divisor. In connection with the linear system $|\mathcal{L}|$, in \cite{popa}, M. Popa gives examples of base points for $|\mathcal{L}|$ on the moduli space $SU(r,\mathcal{O}_C)$ and he proves that for sufficiently large $r$ the base locus is positive dimensional. Moreover, he shows that for a line bundle $L$ with degree $deg(L)\geq 2g+2$ the vector bundle $M_L$ does not admits theta divisor. 
\vskip2mm
\noindent In this work we give new examples of stable syzygy bundles that admit theta divisor over general curves, these examples are contained in Theorems  \ref{thetadivisorLZ} and \ref{principaltheorem} in section 2 below. As corollary of these theorems we have the following result (see Theorem \ref{tetadivisorsemistable} below):

\begin{theorem}\label{introsemistable} Let $C$ be a general curve of genus g, and let $L\in Pic^d(C)$ be a generated line bundle with $h^0(L)=r+1$ such that $M_L$ is strictly  semistable, then $M_L$ admits  reducible theta divisor.
\end{theorem}



In the study of the cohomological stability of the syzygy bundle $M_L$, in \cite{ein} L. Ein and  R. Lazarsfeld prove that $M_L$ is cohomologically stable for any line bundle $L$ on a curve of genus $g$ assuming that $deg(L)\geq 2g+1$. In this context, we prove that over a general curve $C$, $M_L$ is cohomologically semistable for any line bundle $L$ on $C$ that induces a birational map (see Theorem \ref{CohoM} in section 3 below). We find also precise conditions for the cohomological semistability of $M_L$, and these conditions agree with the stability conditions for $M_L$ (see Corollary \ref{condicionescoh} below):

\begin{corollary}\label{intropropiedades}
 Let $L\in Pic^d(C)$ be a line bundle with $h^0(L)=r+1$, inducing a birational morphism  over a general curve $C$ of genus $g\geq 2$. Then 
\begin{enumerate}
\item [(I)] $M_L$ is cohomologically semistable (not stable) if and only if all the following three conditions are satisfied
 \begin{enumerate}
  \item $h^1(L)=0.$
  \item $d=g+r$ and $r$  divides $g$.
  \item There is a line bundle $\mathcal{A}$ with degree $\text{deg}(\mathcal{A})=g+r-1-\frac{g}{r}$ and $h^0(\mathcal{A})=r$ such that $h^0(\wedge^{r-1}M_L\otimes \mathcal{A})\neq 0$. Moreover, in this case we have $h^0(\wedge^{r-1}M_L\otimes \mathcal{A})=1$.
 \end{enumerate}
 
 \item [(II)] $M_L$ is cohomologically stable if and only if $M_L$ is stable.
 \end{enumerate}
\end{corollary}

This paper is organized as follow. In Section 2 we prove the Theorem \ref{introsemistable}, and in Section 3 we prove that over a general curve, the syzygy bundle $M_L$ is cohomologically semistable when $L\in\text{Pic}(C)$ induces a birational map, and we finish with the proof of Corollary \ref{intropropiedades}.  
		
\section{syzygy bundles and theta divisors}

\noindent Let  $L\in Pic^d(C)$ be a globally generated line bundle  over a  general curve $C$ of genus $g$  with $h^0(L)=r+1$. In (\cite{abelyhugo}, Corollary 4.3), the authors proved that the syzygy bundle  $M_L$ is semistable (not stable) if and only if all the following three conditions hold:
\begin{enumerate}
  \item $h^1(L)=0$.
	  
	  \vspace{.1cm}
	  
	  \item $d=g+r$ and $r$ divides  $g$.
	  
	  \vspace{.1cm}
	  
	  \item There is an effective divisor $Z$ with $h^0(L(-Z))=r$ and $\text{deg}(Z)=1+\frac{g}{r}$.
\end{enumerate}
From now on we will denote by $Z$ the effective divisor that satisfy condition (3) above. Furthermore, with the above hyphotesis on $L$ and $Z$ we will denote $\xi:=K_C\otimes L^{\vee}(Z)$.

\begin{remark}
\begin{em}
The line bundle $\xi$ is globally generated: First note that $\text{deg}(\xi)=g-1-r+\frac{g}{r}$ and $h^0(\xi)=h^1(L(-Z))=\frac{g}{r}$. By Riemann-Roch and Serre duality Theorems, we have that for any $p\in C$, the condition $h^0(\xi(-p))=h^0(\xi)-1$ is equivalent to $h^0(L(-Z+p))=r$. From the fact that $h^0(L(-Z))=r$, then $h^0(L(-Z+p))$ is equal to $r$ or $r+1$. Suppose that  $h^0(L(-Z+p))=r+1$ for some point $p\in C$. Since $C$ is a general curve and $L(-Z+p)$ is a line bundle of degree $g+r-\frac{g}{r}$ with  $r+1$ sections, we conclude that the corresponding Brill-Noether number $\rho$ for $L(-Z+p)$ is nonnegative, however $\rho=g-(r+1)\cdot(r-(g+r-\frac{g}{r})+g))=g-(r+1)\cdot \frac{g}{r} <0$, which contradicts that $C$ is general. So $h^0(L(-Z+p))=r$  and $\xi$ is globally generated.
\end{em}
\end{remark}

\begin{theorem}\label{thetadivisorLZ} Let $C$ be a general curve of genus $g\geq 3$.  Under all three conditions above (1)-(3), the syzygy bundle $M_{L(-Z)}$ admits theta divisor.
\end{theorem}

\begin{proof}
Note that by Riemann-Roch Theorem and following Remark 2.1 it is easy to see that $L(-Z)$ is generated, then from the evaluation map $H^0(L(-Z))\otimes\mathcal O_C\to L(-Z)$ we have for every line bundle $\mathcal{P}$ the following exact sequence:
 \begin{eqnarray}\label{sucesion21}
  \xymatrix{ 0 \ar[r]^{} &  M_{L(-Z)}\otimes \mathcal{P} \ar[r]^{}&  H^0(L(-Z))\otimes \mathcal{P}  \ar[r]^{} &  L(-Z)\otimes \mathcal{P} \ar[r]^{} & 0.}
 \end{eqnarray}
The bundle $M_{L(-Z)}$ admits theta divisor if there exists a line bundle $\mathcal{P}$ of degree  $\text{deg}(\mathcal{P})=g+\frac{g}{r}$ such that  $h^0(M_{L(-Z)}\otimes \mathcal{P})=0$. The condition $h^0(M_{L(-Z)}\otimes \mathcal{P})=0$ is equivalent to that the multiplication map $\mu_{\mathcal{P},L(-Z)}:H^0(\mathcal{P})\otimes H^0(L(-Z))\rightarrow H^0(L(-Z)\otimes \mathcal{P})$ is injective.
 Note that the degree $\text{deg}(L(-Z)\otimes \mathcal{P})=2g+r-1\geq 2g $ and $h^0(L(-Z)\otimes \mathcal{P})=g+r$, so we describe the theta divisor $\Theta_{M_{L(-Z)}}$ analyzing the following three cases:
 
\vskip3mm
\noindent (i). If $h^0(\mathcal{P})>\frac{g}{r}+1$ we have $ h^0(\mathcal{P})\cdot h^0(L(-Z))>h^0(\mathcal{P}\otimes L(-Z)),$ thus $\mathcal{P}\in \Theta_{M_{L(-Z)}}.$
  
\vspace{.2cm}
  
\noindent (ii). Assume that $h^0(P)=\frac{g}{r}+1$ and that the linear system  $|\mathcal{P}|$ has a base point. The spaces $ H^0(\mathcal{P})\otimes H^0(L(-Z))$ and $H^0(\mathcal{P}\otimes L(-Z))$ have the same dimension $g+r$, then if  $\mu_{\mathcal{P},L(-Z)}$ is injective, it is an isomorphism, thus the linear system  $|L(-Z)\otimes\mathcal{P}|$ has a base point; but this is impossible since $\text{deg}(L(-Z)\otimes\mathcal{P})=2g+r-1\geq 2g$. Hence we have $\mathcal{P}\in \Theta_{M_{L(-Z)}}$.
  
\vspace{.2cm}
  
\noindent(iii). Finally assume that $|\mathcal{P}|$ is base-point free and $h^0(\mathcal{P})=\frac{g}{r}+1$. By dimension of domain and target space, the  surjectivity and the injectivity of the multiplication map $\mu_{\mathcal{P},L(-Z)}$ are equivalent. So $\mathcal{P}\notin\Theta_{M_{L(-Z)}}$ if and only if  the multiplication map $\mu_{\mathcal{P},L(-Z)}$ is injective or surjective.
 \vskip3mm

\noindent Let $G$ be a general effective divisor of degree $r+1$ over $C$. We recall that $\xi=K_C\otimes L^{\vee}(Z)$ is generated. Let $\mathcal{P}:= \xi(G)$. Since the degree of $\mathcal{P}$ is $g+\frac{g}{r}$ and $G$ is a general divisor we conclude  that $h^0(L(-Z-G))=0$. By Serre duality $h^1(\mathcal{P})=h^0(L(-Z-G))=0$, hence $h^0(\mathcal{P})=\frac{g}{r}+1.$ Moreover, since $G$ is general of degree $r+1$ we have $h^0(K_C(-\mathcal{P}+p))=0$ for any $p\in C$ and so $h^0(\mathcal{P}-p)=\frac{g}{r}$, thus $\mathcal{P}$ is free of base points. 

\noindent To finish the proof we are going to show that the multiplication map  $\mu_{\mathcal{P},L(-Z)}:H^0(\mathcal{P})\otimes H^0(L(-Z))\rightarrow H^0(K_C(G))$ is injective. We have $h^0(\mathcal{P}(-G))=h^0(\xi)=\frac{g}{r}=h^0(\mathcal{P})-1$. Since $\mathcal{P}$ is generated, there exists a basis $\{\sigma_0, \sigma_1,\ldots,\sigma_{\frac{g}{r}}\}\subset H^0(\mathcal{P})$ such that for every $j\in\{1,\ldots,\frac{g}{r}\}$ and for all $p\in G$, $\sigma_j(p)=0$ and $\sigma_0(p)\neq 0$. 

 \noindent Let $t:=l_0\otimes \sigma_0+l_1\otimes \sigma_1+\ldots+l_{\frac{g}{r}}\otimes \sigma_{\frac{g}{r}}\in\text{Ker}(\mu_{\mathcal{P}, L(-Z)})$, where $l_0,l_1,\ldots,l_{\frac{g}{r}}$ are sections in $H^0(L(-Z))$, then $l_0\cdot\sigma_0+l_1\cdot\sigma_1+\ldots+l_{\frac{g}{r}}\cdot\sigma_{\frac{g}{r}}=0$, that is, $-l_0\cdot\sigma_0=l_1\cdot\sigma_1+\ldots+l_{\frac{g}{r}}\cdot\sigma_{\frac{g}{r}}$. Note that for any $p\in G$, we have $l_0(p)=0$ because  $\sigma_0(p)\neq 0$ and $\sigma_j(p)=0$ for $j=1,...,\frac{g}{r}$. Since $l_0\in H^0(L(-Z)) $ and $l_0$ vanish on $G$, it follows that $l_0\in H^0(L(-Z-G))=\{0\}$, so $t=l_1\otimes \sigma_1+\ldots+l_{\frac{g}{r}}\otimes \sigma_{\frac{g}{r}}$, that is, $t$ is an element in the kernel of the multiplication map
\begin{eqnarray*}
 <\sigma_1,\ldots,\sigma_{\frac{g}{r}}>\otimes H^0(L(-Z))\xrightarrow{\mu}H^0(K_C(G)),
\end{eqnarray*}
where $< \sigma_1,\ldots,\sigma_{\frac{g}{r}}>=H^0(\mathcal{P}(-G))=H^0(\xi)=H^0(K_C\otimes L^{\vee}(Z))$. Note also that $\text{image}(\mu)\subset H^0(K_C)$, in fact, $\mu=\mu_{L(-Z),\xi}$ is the Petri map for $L(-Z)$ and we have the following diagram

$$
\begin{tikzcd}
H^0(\xi)\otimes H^0(L(-Z))\arrow{r}{\mu_{L(-Z),\xi}} \arrow[hookrightarrow]{d}
&H^0(K_C) \arrow[hookrightarrow]{d}\\
H^0(\mathcal{P})\otimes H^0(L(-Z))\arrow{r}{\mu_{\mathcal{P},L(-Z)}}&H^0(K_C(G)) 
\end{tikzcd}
$$

\noindent Since $C$ is general, the Petri map $\mu:H^0(L(-Z))\otimes H^0( K_C\otimes L^{\vee}(Z))\rightarrow H^0(K_C)$ is injective, then $t=0$ and this implies that $\mu_{\mathcal{P}, L(-Z)}$ is injective, thus by condition (iii) we have $\mathcal{P}\notin \Theta_{M_{L(-Z)}}$, hence $M_{L(-Z)}$ admits theta divisor.
\end{proof}

We recall by Remark 2.1 that $\xi$ is globally generated, then following the ideas in the proof of Theorem \ref{thetadivisorLZ} we have:

\begin{theorem}\label{principaltheorem}
 Let $L\in Pic^d(C)$ be a globally generated line bundle over a general curve $C$ such that $M_L$ is semistable (not stable). The syzygy bundle $M_{\xi}$ is stable and admits theta divisor.     

\end{theorem}

\begin{proof}
By Serre duality Theorem, $h^1(\xi)=h^0(L(-Z))=r\neq 0$, this condition implies by (\cite{abelyhugo}, Corollary 4.3) that $M_{\xi}$ is stable. The slope $\mu(M_{\xi})=-\frac{g-1-r+\frac{g}{r}}{\frac{g}{r}-1}=-1-r$ is integral, so $M_{\xi}$ admits theta divisor if there exists a line bundle $\mathcal{F}$ of degree $\text{deg}(\mathcal{F})=g+r$ such that  $h^0(M_{\xi}\otimes \mathcal{F})=0.$ From the exact sequence
 \begin{eqnarray*}
  \xymatrix{ 0 \ar[r]^{} &  M_{\xi}\otimes \mathcal{F} \ar[r]^{}&  H^0(\xi)\otimes \mathcal{F}  \ar[r]^{} &  \xi\otimes \mathcal{F} \ar[r]^{} & 0,}
 \end{eqnarray*}
the condition $h^0(M_{\xi}\otimes \mathcal{F})=0$ it is equivalent to that the multiplication map $\mu_{\mathcal{F},\xi}:H^0(\mathcal{F})\otimes H^0(\xi)\rightarrow H^0(\xi\otimes \mathcal{F})$ is injective. Let $D$ be a general effective divisor of degree $\frac{g}{r}+1$,  and take $\mathcal{F}:= L(-Z+D)$. Note that the degree of $\mathcal{F}$ is $g+r$, and since $D$ is general we conclude  that $H^0(\xi(-D))=0$. By Serre duality $h^1(\mathcal{F})=h^0(\xi(-D))=0$, hence $h^0(\mathcal{F})=r+1.$ Moreover, $\mathcal{F}$ is generated: Since $D$ is general of degree $1+\frac{g}{r}$ we have $h^0(K_C(-\mathcal{F}+p))=0$ for any $p\in C$ and so $h^0(\mathcal{F}-p)=r$ as required.  Using the same argument as in the proof of Theorem \ref{thetadivisorLZ}, we have $\mu_{\mathcal{F},\xi}:H^0(\mathcal{F})\otimes H^0(\xi)\rightarrow H^0(K_C(D))$ is injective, and this completes the proof.
\end{proof}

\begin{theorem}\label{tetadivisorsemistable} 
Let $L\in Pic^d(C)$ be a generated line bundle over a general curve of genus $g$  with $h^0(L)=r+1$ such  that $M_L$ is strictly  semistable, then $M_L$ admits reducible theta divisor.
\end{theorem}
\begin{proof}

We recall by (\cite{abelyhugo}, Corollary 4.3) that if $M_L$ is strictly semistable, then $h^1(L)=0$,  $r$ divides $g$ and there exists an effective divisor $Z$  of degree $1+\frac{g}{r}$ such that $h^0(L(-Z))=r$ and $L(-Z)$ is generated. The vector bundles $M_L$ and $M_{L(-Z)}$ fit in the following exact sequence 
\begin{eqnarray}\label{sucesion22}
\xymatrix{ 0 \ar[r]^{} &  M_{L(-Z)} \ar[r]^{}&  M_L  \ar[r]^{} &  \mathcal{O}(-Z) \ar[r]^{} & 0.}
 \end{eqnarray}
 By Theorem \ref{thetadivisorLZ}, $M_{L(-Z)}$ admits  theta divisor and since the three vector bundles in the exact sequence  (\ref{sucesion22}) have the same slope $-1-\frac{g}{r}$, we see that $M_L$ admits a theta divisor, and in this case 
 $\Theta_{M_L}=\Theta_{M_{L(-z)}}\cup\Theta_{\mathcal O(-Z)}$ as divisors. This completes the proof.
\end{proof}

\section{Cohomologically semistable}
 In this section we prove that if $C$ is general then the syzygy bundle $M_L$ is cohomologically semistable, and we find precise conditions for the cohomologically semistable of $M_L$. Moreover, we show that the (semi)stability of $M_L$ is equivalent to the cohomological (semi) stability of $M_L$ when $C $ is a general curve and $L$ induces a birational morphism.

\begin{definition}
Let $E$ be a vector bundle over a curve $C$. We say that $E$ is cohomologically stable (respectively cohomologically semistable) if for any line bundle $P$ of degree $p$ and for any  integer $t<rank(E)$ we have
\begin{equation*}
h^0(\wedge^t E\otimes P^{\vee})=0
\end{equation*}
whenever $p\geq t\mu (E)$ (respectively, $p>t\mu(E))$.
\end{definition}
 
  \begin{remark}
\begin{em} From the isomorphism $Pic^s(C)\rightarrow Pic^{-s}(C)$ which associates to a line bundle its dual, and since the degree of $M_L$ is $-\frac{d}{r}$, it follows that $M_L$ is cohomologically semistable (respectively stable) if and only if  for any integer $t<r$ and any line bundle $\mathcal{A}$ of degree $deg(\mathcal{A})<t\cdot \frac{d}{r}$ $( respectively \leq)$, $h^0(\wedge^t M_L\otimes \mathcal{A})=0.$
\end{em}
\end{remark}
 
  \begin{remark}\label{cohomimplicaestabilidad}
  \begin{em}
  First note that cohomological (semi)stability implies (semi)stability: suppose that $E$ is not (semi)stable, so there exist a subbundle $F\subset E$ of degree $p$ and rank $t$ with $p>t\mu(E)$ (respectively $p\geq t\mu(E)$). Consider the non zero morphism $det(F)\rightarrow \wedge^t E$ induced by $F$, we have $h^0(\wedge^t E\otimes det(F)^{\vee})\neq 0$, and thus $E$ is not cohomologically (semi)stable.
 
  \end{em}
  \end{remark}
 \noindent The proof of the following lemma is a consequence of (\cite{mistretastopino}, Lemma 7.4.) 
\begin{lemma}
 Let $L$ be a line bundle over $C$ with $h^0(C,L)=r+1$ inducing a birational morphism $C\dashrightarrow\mathbb P^r$, then there exist general points $x_1,\ldots,x_{r-1}\in C$ such that $M_L$ fits in the following exact sequence
 \begin{eqnarray}
  \xymatrix{ 0 \ar[r]^{} &   L^{\vee}(\sum\limits_{j=1}^{r-1} x_j) \ar[r]^{}&  M_L \ar[r]^{} &  \bigoplus\limits_{j=1}^{r-1}\mathcal{O}_C (-x_j)\ar[r]^{} & 0.}
 \end{eqnarray}
\end{lemma}
 \begin{proposition}\label{cohomology}Let $L\in Pic^d(C)$ be a line bundle with $h^0(L)=r+1$ inducing a birational morphism over a curve $C$ of genus $g$. Let $\mathcal{A}$ be a line bundle over $C$ such that $\text{deg}(\mathcal{A})\leq t\cdot \frac{d}{r} $ and $h^0(\mathcal{A})\leq t$ with integers $t$, $d$ and $r$ satisfying $0<t<r<d\leq g+r$, then $h^0(\wedge^t M_L \otimes \mathcal{A})=0$.
 \end{proposition}

\begin{proof}
 Since $L$ is generated we have the exact sequence $$0\to M_L\to H^0(L)\otimes \mathcal{O}_C\to L\to 0.$$ 
 By Lemma 3.4, there exist general points $x_1,\ldots,x_{r-1}$ over $C$ such that $M_L$ fits into following exact sequence
  \begin{eqnarray*}
  \xymatrix{ 0 \ar[r]^{} &   L^{\vee}(\sum\limits_{j=1}^{r-1} x_j) \ar[r]^{}&  M_L \ar[r]^{} &  \bigoplus\limits_{j=1}^{r-1}\mathcal{O}_C (-x_j)\ar[r]^{} & 0,}
 \end{eqnarray*}
and for any $i_1,\ldots,i_t\in \{ 1,\ldots,r-1\}$ with $1\leq i_1<\ldots<i_t\leq r-1$ we have $h^0(\mathcal{A}(-x_{i_1}-\ldots-x_{i_t}))=0$. Put $F=\bigoplus\limits_{j=1}^{r-1}\mathcal{O}_C(-x_j)$. We get the exact sequence of exterior powers
\begin{eqnarray*}
  \xymatrix{ 0 \ar[r]^{} & \wedge^{t-1}F\otimes L^{\vee}(\sum\limits_{j=1}^{r-1}x_j) \ar[r]^{}& \wedge^t M_L \ar[r]^{} &  \wedge^{t}F \ar[r]^{} & 0.}
 \end{eqnarray*}
 That is,
 \begin{eqnarray}\label{cohomologyM}
  \xymatrix{ 0 \ar[r]^{} & \bigoplus\limits_{1\leq i_1< \ldots< i_{t-1}\leq r-2}   L^{\vee}(\sum\limits_{j=1}^{r-1}x_j-\sum\limits_{j=1}^{t-1}x_{i_j}) \ar[r]^{}& \wedge^t M_L \ar[r]^{} &  \bigoplus\mathcal{O}_C (-\sum\limits_{j=1}^{t}x_{i_j})\ar[r]^{} & 0.}
 \end{eqnarray}
 Twisting the exact sequence (\ref{cohomologyM}) by $\mathcal{A}$, we get the following exact sequence in cohomology
 \begin{eqnarray}\label{coh3}
  \xymatrix{  \bigoplus H^0(L^{\vee}(\sum\limits_{j=1}^{r-1}x_j-\sum\limits_{j=1}^{t-1}x_{i_j})\otimes \mathcal{A}) \ar[r]^{}& H^0(\wedge^t M_L\otimes \mathcal{A}) \ar[r]^{} &  \bigoplus H^0(\mathcal{O}_C (-\sum\limits_{j=1}^{t}x_{i_j})\otimes \mathcal{A}) &}
 \end{eqnarray}
 We want to prove that both left and right sides in (\ref{coh3}) are trivial. The right-hand side is a sum of global sections of line bundles such that for any $i_1,\ldots,i_t\in \{ 1,\ldots,r-1\}$ with $1\leq i_1<\ldots<i_t\leq r-1$ we have $h^0(\mathcal{A}(-x_{i_1}-\ldots-x_{i_t}))=0$, and  thus is equal to zero. The left-hand side is a sum of global sections of line bundles of degree
 \begin{eqnarray*}
  \text{deg}(  L^{\vee}(\sum_{j=1}^{r-1}x_j-\sum_{j=1}^{t-1}x_{i_j}) \otimes \mathcal{A})&=&-d+(r-1)-(t-1)+\text{deg}(\mathcal{A})\\
 &\leq& -d+r-t+t\cdot\frac{d}{r}\\
 &=&(r-t)(1-\frac{d}{r}). 
 \end{eqnarray*}
Since $r-t>0$ and $d>r$, we have this degree is negative and we get $ H^0(L^{\vee}(\sum_{j=1}^{r-1}x_j-\sum_{j=1}^{t-1}x_{i_j})\otimes \mathcal{A})=0$. Hence $H^0(\wedge^t M_L\otimes \mathcal{A})=0$ and this completes the proof.
	\end{proof}

On a general curve we can bound the dimension of the space global sections of the line bundle $\mathcal{A}$ satisfying the properties of Proposition \ref{cohomology}, that is:

\begin{proposition}\label{seccionesA}
   Let $\mathcal{A}$ be a line bundle over a general curve $C$ of genus $g$ such that $\text{deg}(\mathcal{A})\leq t\cdot \frac{d}{r} $ with integers $t$, $d$ and $r$ satisfying $0<t<r<d\leq g+r$, then $h^0(\mathcal{A})\leq t+1$. Moreover, if $h^0(\mathcal{A})=t+1$ then $\text{deg}(\mathcal{A})=t\cdot \frac{d}{r}$, $d=g+r$ and $t+1=r$.
\end{proposition}

\begin{proof} 
Suppose that $h^0(\mathcal{A})\geq t+2$. Since $C$ is a general curve and $\mathcal{A}$ is a line bundle of degree $\text{deg}(\mathcal{A})$ with at least $t+2$ sections, we conclude that the corresponding Brill-Noether number $\rho$ for the line bundle $\mathcal{A}$ is nonnegative, but 
\begin{eqnarray*}
\rho&:=&g-(t+2)\cdot (t+1-\text{deg}(\mathcal{A})+g)\\
       &=&g-(t+2)(t+1)+\text{deg}(\mathcal{A})(t+2)-(t+2)g\\
       &\leq&-g-(t+2)(t+1)+t(t+2)(\frac{g+r}{r})-tg\\
       &=&-g-(t+2)-tg(1-\frac{t+2}{r})<0, 
\end{eqnarray*}
which contradicts that $C$ is general. Then $h^0(\mathcal{A})\leq t+1$ and this proves the first statement. Now, suppose that  $h^0(\mathcal{A})=t+1$. Since $C$ is a general curve and $\mathcal{A}$ is a line bundle of degree $\text{deg}(\mathcal{A})$ with at least $t+1$ sections, we conclude that the corresponding Brill-Noether number $\rho$ for the line bundle $\mathcal{A}$ is nonnegative, but 
\begin{eqnarray*}
0\leq \rho&:=&g-(t+1)\cdot (t-\text{deg}(\mathcal{A})+g)\\
       &=&-t(t+1)+\text{deg}(\mathcal{A})(t+1)-tg\\
       &\leq& t(t+1)\frac{d}{r}-t(t+1)-tg\\
       &\leq& t(t+1)(\frac{g+r}{r})-t(t+1)-tg\\
       &=&gt(\frac{t+1}{r}-1)\leq 0. 
\end{eqnarray*}
So the Brill-Noether number $\rho$ is equal to zero if and only if  $\text{deg}(\mathcal{A})=t\cdot \frac{d}{r}$, $d=g+r$ and $t+1=r$.
 \end{proof}

 \begin{theorem}\label{CohoM}
   Let $L\in Pic^d(C)$ be a line bundle with $h^0(L)=r+1$, inducing a birational morphism over a general curve $C$ of genus $g$. Then $M_L$ is cohomologically semistable.
 \end{theorem}

\begin{proof}Let $t<r$ be a positive integer, and consider a line bundle $\mathcal{A}$  of degree $\text{deg}(\mathcal{A})<t\cdot \frac{d}{r}$. From Propositions  \ref{cohomology} and \ref{seccionesA} we have $h^0(\mathcal{A})\leq t$ and $h^0(\wedge^t M_L \otimes \mathcal{A})=0$. Hence $M_L$ is cohomologically semistable, and which completes the proof.
 \end{proof}

It is of interest to find precise conditions for the  cohomological semistability of  $M_L$. In this context we have the following result:

 \begin{corollary}\label{hdivide}
 Let $L\in Pic^d(C)$ be a line bundle with $h^0(L)=r+1$, inducing a birational morphism  over a general curve $C$ of genus $g\geq 2$. Then $M_L$ is cohomologically stable if one the following conditions is satisfied:
\vskip1mm
 
   (i) $h^1(L)\neq 0$, or
  
  (ii) $h^1(L)=0$ and $r$ not divides $g$.
\end{corollary}
\begin{proof}
Set $h:=h^1(L)$. Let $t$ be an integer positive with $t<r$ and let $\mathcal{A}$ be a line bundle of degree $\text{deg}(\mathcal{A})\leq t\cdot \frac{d}{r}$. By Proposition \ref{cohomology}, it is enough to prove that $h^0(\mathcal{A})\leq t$. From Proposition \ref{seccionesA}, we have $h^0(\mathcal{A})\leq t+1$. In particular, if $h^0(\mathcal{A})= t+1$ then $t=r-1$, $d=g+r$ and $deg(\mathcal{A})=t\cdot \frac{d}{r}=(r-1)\cdot \frac{g+r}{r}.$ Under our hypothesis, we will prove that these three conditions are not satisfied simultaneously.
\vskip1mm

\noindent If $h\neq 0$, then $d <g+r$, and from Proposition \ref{seccionesA}  we have $h^0(\mathcal{A})\leq t$. This proves $(i)$.

\noindent To prove $(ii)$, if $h=0$ and $t=r-1$, then $d=g+r$ and 
\begin{eqnarray*}
 t\cdot \frac{d}{r}=(r-1)\cdot (\frac{g+r}{r})=g+r-1-\frac{g}{r}.
\end{eqnarray*}
 Suppose that $r$ not divides $g$, so the condition $\text{deg}(\mathcal{A})\leq (r-1)\cdot\frac{d}{r} $ implies that $\text{deg}(\mathcal{A})< (r-1)\cdot\frac{d}{r}$, and hence $h^0(\mathcal{A})\leq t=r-1$. This proves $(ii)$.
\end{proof}

\begin{remark}
\begin{em}
If $E$ is semistable then $E$ is cohomologically semistable: let $\mathcal{A}$ be a line bundle of degree $p$ and for any integer $t<\text{rank}(E)$ with $p<-t\mu(E)$. Since any exterior power of semistable bundle is semistable, we see that $\wedge^t E\otimes \mathcal{A}$ is semistable of slope negative and so $h^0(\wedge^t E\otimes \mathcal{A})=0.$ By Remark \ref{cohomimplicaestabilidad}, it follows that the cohomological semistability is equivalent to semistability. Nevertheless, since the exterior power of stables are not necessarily stable, we have the cohomologically stability is a condition stronger than stability. 

\noindent In this direction, we find precise conditions for strictly cohomological semistability of $M_L$ and prove that this conditions agree with semistability conditions for $M_L$. As consequence, it follows that the cohomological stability of $M_L$ is equivalent to stability of $M_L$.

\end{em}
\end{remark}

\noindent Let $L$ be a globally generated line bundle over a general curve of genus $g$  with $h^0(L)=r+1$. In (\cite{abelyhugo}, Corollary 4.3) the authors prove that the syzygy bundle $M_L$ is semistable (not stable) if and only if all the following three conditions hold:
\begin{enumerate}
  \item[(1)] $h^1(L)=0$.
	  
	  \vspace{.1cm}
	  
	  \item[(2)] $d=g+r$ and $r$ divides  $g$.
	  
	  \vspace{.1cm}
	  
	  \item[(3)] There is an effective divisor $Z$ with $h^0(L(-Z))=r$ and $\text{deg}(Z)=1+\frac{g}{r}$.
\end{enumerate}


\begin{corollary}\label{condicionescoh}
 Let $L\in Pic^d(C)$ be a line bundle with $h^0(L)=r+1$, inducing a birational morphism  over a general curve $C$ of genus $g\geq 2$. Then 
\begin{enumerate}
\item [(I)] $M_L$ is cohomologically semistable (not stable) if and only if all the following three conditions are satisfied
 \begin{enumerate}
  \item $h^1(L)=0.$
  \item $d=g+r$ and $r$  divides $g$.
  \item There is a line bundle $\mathcal{A}$ with degree $\text{deg}(\mathcal{A})=g+r-1-\frac{g}{r}$ and $h^0(\mathcal{A})=r$ such that $h^0(\wedge^{r-1}M_L\otimes \mathcal{A})\neq 0$.
 \end{enumerate}
 Moreover, in this case we have $h^0(\wedge^{r-1}M_L\otimes \mathcal{A})=1$.
 \item [(II)] $M_L$ is cohomologically stable if and only if $M_L$ is stable.
 \end{enumerate}
\end{corollary}

\begin{proof}
\begin{enumerate}
\item[(I)] 

\noindent($\Rightarrow$) Suppose that $M_L$ is strictly cohomologically semistable. By Corollary \ref{hdivide}, we have $h^1(L)=0$ and $r$ divides $g$, then by Riemann-Roch Theorem, $d=g+r$, thus we have conditions $(a)$ and $(b)$. Moreover,  there exist an integer $t<r$ and a line bundle $\mathcal{A}$ of degree $deg(\mathcal{A})=t\cdot \frac{d}{r}$ such that $h^0(\wedge^t M_L\otimes \mathcal{A})\neq 0$.  From the Propositions \ref{cohomology} and \ref{seccionesA},  we get that $t=r-1$, $h^0(\mathcal{A})=t+1=r$ and $\text{deg}(\mathcal{A})=t\cdot \frac{d}{r}=(r-1)(\frac{g+r}{r}).$ Consider general points $x_1,\ldots,x_{r-1}$ such that $h^0(\mathcal{A}(-x_{1}-\ldots-x_{r-1}))=1$, then by Lemma 3.4 we have the following exact sequence
 \begin{eqnarray}
  \xymatrix{ 0 \ar[r]^{} &   L^{\vee}(\sum\limits_{j=1}^{r-1} x_j) \ar[r]^{}&  M_L \ar[r]^{} &  \bigoplus\limits_{j=1}^{r-1}\mathcal{O}_C (-x_j)\ar[r]^{} & 0.}
 \end{eqnarray}
 Put $F=\bigoplus\limits_{j=1}^{r-1}\mathcal{O}_C(-x_j)$. We have the exact sequence of exterior powers
\begin{eqnarray*}
  \xymatrix{ 0 \ar[r]^{} & \wedge^{r-2}F\otimes L^{\vee}(\sum\limits_{j=1}^{r-1}x_j) \ar[r]^{}& \wedge^{r-1} M_L \ar[r]^{} &  \wedge^{r-1}F \ar[r]^{} & 0,}
 \end{eqnarray*}
 that is,
 \begin{eqnarray}\label{cohomologyMM}
  \xymatrix{ 0 \ar[r]^{} & \bigoplus\limits_{j=1}^{r-1}   L^{\vee}(-x_j) \ar[r]^{}& \wedge^{r-1} M_L \ar[r]^{} &  \mathcal{O}_C (-\sum\limits_{j=1}^{r-1}x_{j})\ar[r]^{} & 0.}
 \end{eqnarray}
 Twisting the exact sequence (\ref{cohomologyMM}) by $\mathcal{A}$, we get the following exact sequence in cohomology
 \begin{eqnarray*}
  \xymatrix{   & \bigoplus\limits_{j=1}^{r-1} H^0(L^{\vee}\otimes \mathcal{A}(-x_j)) \ar[r]^{}& H^0(\wedge^{r-1} M_L\otimes \mathcal{A}) \ar[r]^{} &   H^0(\mathcal{A}(-\sum\limits_{j=1}^{r-1}x_{j})). &}
 \end{eqnarray*}
The dimension of the right-hand side is one and the left-hand side is a sum of global sections of line bundles of degree $\text{deg}(L^{\vee}\otimes \mathcal{A}(-x_j))=-2-\frac{g}{r}<0,$ it follows that $h^0(\mathcal{A}(-\sum\limits_{j=1}^{r-1}x_{j}))=0$ and $h^0(\wedge^{r-1}M_L\otimes \mathcal{A})=1$. This proves condition $(c)$.

\vspace{.3cm}

\noindent($\Leftarrow$) The conditions $(a)$ and $(b)$ implies that $\mathcal{A}$ is a line bundle of degree 
\begin{eqnarray*}
\text{deg}(\mathcal{A})=g+r-1-\frac{g}{r}=(r-1)\cdot(\frac{g+r}{r})=(r-1)\cdot\frac{d}{r}. 
\end{eqnarray*}

Using Theorem  \ref{cohomologyM} and the hypothesis $h^0(\wedge^{r-1}M_L \otimes \mathcal{A})\neq 0$, it follows that $M_L$ is strictly cohomologically semistable.

\vspace{.3cm}

\item[(II)]
\noindent $(\Rightarrow)$ This follows from Remark \ref{cohomimplicaestabilidad}, however we want to give another proof by applying the conditions (1)-(3) and (a)-(c) above. So, suppose that $M_L$ is not stable. Since the conditions $(1)$ and $(2)$ are equivalents to $(a)$ and $(b)$ respectively, we can assume that condition $(3)$ is not satisfied. So there is an effective divisor $Z$ of degree $\text{deg}(Z)=1+\frac{g}{r}$ and $h^0(L(-Z))=r$. We recall that $L(-Z)$ is generated, and the vector bundles $M_L$ and $M_{L(-Z)}$ fit in the following exact sequence 
 \begin{eqnarray*}
\xymatrix{ 0 \ar[r]^{} &  M_{L(-Z)} \ar[r]^{}&  M_L  \ar[r]^{} &  \mathcal{O}(-Z) \ar[r]^{} & 0.}
 \end{eqnarray*}	
Taking exterior powers, we obtain
 \begin{eqnarray*}
\xymatrix{ 0 \ar[r]^{} &  L^{\vee}(Z) \ar[r]^{}&  \wedge^{r-1}M_L  \ar[r]^{} & \wedge^{r-2}M_{L(-Z)}\otimes \mathcal{O}(-Z) \ar[r]^{} & 0.}
 \end{eqnarray*}
 twisting by $L(-Z)$, we get
  \begin{eqnarray*}
\xymatrix{ 0 \ar[r]^{} &  \mathcal{O}_C \ar[r]^{}&  \wedge^{r-1}M_L\otimes L(-Z)  \ar[r]^{} & \wedge^{r-2}M_{L(-Z)}\otimes L(-2Z) \ar[r]^{} & 0.}
 \end{eqnarray*}
So $h^0(\wedge^{r-1}M_L\otimes L(-Z))\neq 0$, and $M_L$ is strictly  cohomologically semistable.

\vspace{.2cm}

\noindent$(\Leftarrow)$ Suppose that $M_L$ is not cohomologically stable, since the conditions $(a)$ and $(b)$ are equivalents to $(1)$ and $(2)$ respectively, we can assume that condition $(c)$ is not satisfied. So there is a line bundle $\mathcal{A}$   of degree $g+r-1-\frac{g}{r}$ such that $h^0(\wedge^{r-1}M_L\otimes \mathcal{A})\neq 0$. Using the following isomorphism 
 \begin{eqnarray*}
  \wedge^{r-1}M_L\cong M_L^{\vee}\otimes L^{\vee}
 \end{eqnarray*}
of vector bundles, there exists a not zero morphism $\mathcal{A}^{\vee}\otimes L\rightarrow M_L^{\vee}$. Since the slopes of $\mathcal{A}^{\vee}\otimes L$ and $M^{\vee}_L$ are the same and   $\mathcal{A}^{\vee}\otimes L$ is a line bundle, it follows that $M_L^{\vee}$ can not be stable.

\vspace{.2cm}

\end{enumerate}
\end{proof}

\noindent{\bf Acknowledgments:} Second named author warmly thanks the Centro de Ciencias Matem\'aticas (CCM,UNAM) in Morelia City for their hospitality and for the use of their resources during a posdoctoral year.  Second named author is supported by a Posdoctoral Fellowship from CONACyT, M\'exico.

 \end{document}